\documentclass[11pt,reqno]{amsart}
\usepackage{amsmath,amscd,amsfonts,amsthm,amssymb,latexsym,mathrsfs}

\usepackage{showlabels}

\usepackage[all]{xy}
\usepackage{graphicx,epic,eepic,epsfig}

\theoremstyle{plain}
   \newtheorem{theorem}{Theorem}[section]
   \newtheorem{proposition}[theorem]{Proposition}
   \newtheorem{lemma}[theorem]{Lemma}
   \newtheorem{corollary}[theorem]{Corollary}

\theoremstyle{definition}
   
   \newtheorem{example}[theorem]{Example}

\theoremstyle{remark}
   \newtheorem{remark}[theorem]{Remark}

 \numberwithin{equation}{section}

\author[P.~Br\"and\'en]{Petter Br\"and\'en$^\ast$}
\thanks{$^\ast$Royal Swedish Academy of Sciences Research Fellow supported by a grant from the
Knut and Alice Wallenberg Foundation. Partially supported by the G\"oran Gustafsson Foundation.}
\address{Department of Mathematics, Royal Institute of Technology, SE-100 44 Stockholm, Sweden}
\email{pbranden@kth.se}

\author[L.~Moci]{Luca Moci$^\dag$}
\thanks{$^{\dag}$Marie Curie Fellow of Istituto Nazionale di Alta Matematica. Partially supported by %Institut Mittag-Leffler (Djursholm) and
PRIN 2009 ``Spazi di moduli e teoria di Lie"}
\address{IMJ, Universit\'{e} de Paris 7, 175 Rue du Chevaleret, 75013 France}
\email{lucamoci@hotmail.com}

\def\kkk{\kern.2ex\mbox{\raise.5ex\hbox{{\rule{.35em}{.12ex}}}}\kern.2ex}

\newcommand{\kk}{\mathbf{k}}
\newcommand{\vv}{\mathbf{v}}

\newcommand{\NN}{\mathbb{N}}

\newcommand{\PP}{\mathcal{P}}
\newcommand{\ZT}{\mathcal{Z}}
\newcommand{\GG}{\mathcal{G}}

\newcommand{\BB}{\mathcal{B}}

\newcommand{\CT}{\mathcal{C}}
\newcommand{\LL}{\mathcal{L}}

\newcommand{\A}{\mathcal{A}}
\newcommand{\MM}{\mathcal{M}}

\newcommand{\RR}{\mathbb{R}}

\newcommand{\Di}{\mathcal{D}}
\newcommand{\ZZ}{\mathbb{Z}}
\newcommand{\cA}{\mathcal{A}}

\renewcommand{\Im}{{\rm Im}}

\def\newop#1{\expandafter\def\csname #1\endcsname{\mathop{\rm
#1}\nolimits}}

\newop{per}
\newop{diag}
\newop{supp}
\newop{Proj}
\newop{ker}
\newop{rk}
\newop{Sym}
\newop{sign}
\newop{Int}
\newop{LCM}
\newop{Hom}
\newop{St}
\newop{det}
\newop{I}
\newop{E}
\newcommand{\EA}{E}
\newcommand{\IA}{I}
\begin{document}

\title[The multivariate arithmetic Tutte polynomial]{The multivariate arithmetic \\Tutte polynomial}

\maketitle
\thispagestyle{empty}
\begin{abstract}
We introduce an arithmetic version of the multivariate Tutte polynomial, and (for representable arithmetic matroids) a quasi-polynomial that interpolates between the two. A generalized Fortuin-Kasteleyn representation with applications to arithmetic colorings and flows is obtained. We give a new and more general proof of the positivity of the coefficients of the arithmetic Tutte polynomial, and (in the representable case) a geometrical interpretation of them.
\end{abstract}
\tableofcontents
\newpage
\section{Introduction}\label{introsec}

In this paper we introduce and study  a \emph{multivariate arithmetic Tutte polynomial}. Recall that the \emph{Tutte polynomial} is a bivariate polynomial with several well-known specializations: for instance the chromatic polynomial of a graph, or the characteristic polynomial of a hyperplane arrangement can be obtained by specializing the Tutte polynomial. Also, its coefficients are nonnegative, as proved by Crapo by providing an explicit combinatorial interpretation \cite{Cr,Tu}.

Recently an arithmetic version of this polynomial has been studied \cite{MoT, DM}. Namely, to a finite list $\LL$ of elements in a finitely generated abelian group $G$, one associates an \emph{arithmetic Tutte polynomial}. This is a bivariate polynomial that encodes several geometric, algebraic and combinatorial invariants. For instance:
\begin{itemize}
\item the characteristic polynomial of the generalized toric arrangement $\mathcal{T}(\LL)$.
This is a family of submanifolds in the abelian compact Lie group $\Hom(G, \mathbb{S}^1)$ (see \cite[Section 5]{MoT});
\item the dimension of the Dahmen-Micchelli space $DM(\LL)$. This vector space was introduced in order to study vector partition functions (see \cite[Section 6]{MoT});
\item the Ehrhart polynomial of the zonotope $\mathcal{Z}(\LL)$ (see \cite[Section 4]{MoT} and \cite{DM-G}).
\end{itemize}

Furthermore, the arithmetic Tutte polynomial has applications to graph theory \cite{DM-G}, and it has nonnegative coefficients, as proved in \cite{DM} by providing a combinatorial interpretation that extends Crapo's theorem.

The Tutte polynomial is naturally defined in the general framework of \emph{matroids}, while the arithmetic Tutte polynomial is associated to an \emph{arithmetic matroid} $\cA$, which is a matroid equipped with  a \emph{multiplicity function} $m$ satisfying additional axioms.
When $\cA$ is represented by a list of elements in a finitely generated abelian group, the function $m$ encodes arithmetic information, just as  the rank function encodes linear-algebraic information.

Recently Sokal \cite{Sok} studied a  multivariate  generalization of the Tutte polynomial. The \emph{multivariate Tutte polynomial} has a variable $v_e$ for each element $e$ of the matroid (or edge $e$ of the graph), and an extra variable $q$.
If all the variables $v_e$ are set to be equal,
we obtain a bivariate polynomial which is essentially
equivalent to the standard Tutte polynomial.
%But as Sokal showed, it is often useful to consider the multivariate polynomial,
% even if one is ultimately interested in a particular
%two-variable or one-variable specialization.
%For instance, Sokal's polynomial is \emph{multiaffine}
%in the variables $\vv$ (i.e., of degree 1 in each $v_e$ separately);
%often a multiaffine polynomial in several variables is easier to handle
%than a general polynomial in a single variable: indeed, it allows induction on the number of variables, making many proofs much easier.
In the case of graphs, the polynomial is known to physicists as the
\emph{partition function of the $q$-state Potts model},
%\cite{Potts_52} \cite{Ising_25}
which along with the related
\emph{Fortuin--Kasteleyn random-cluster model}
%\cite{Kasteleyn_69,Fortuin_72}
plays an important role in the theory of phase transitions
and critical phenomena.
%
%In these models,
%an ``atom'' (or ``spin'') can exist in any one of
%$q$ different states.
%The energy of a configuration is the sum, over all its atoms,
%of $0$ if the spins at the two endpoints of that edge are unequal
%and a nonzero contribution if they are equal. (The model is said to be \emph{ferromagnetic} if the nonzero contributions are positive, \emph{antiferromagnetic} if they are negative).
%It is then natural to imagine the Potts model as a graph in which every edge is an athom, whose state is represented by a color. Then proper colorings correspond to the configurations with maximal (or minimal) energy (see \cite{Sok} and the wide bibliography therein).
%
In this paper we introduce a \emph{multivariate arithmetic Tutte polynomial}
$$
Z_\cA(q,\vv) := \sum_{A \subseteq E} m(A)q^{-\rk(A)} \prod_{e \in A}v_e.
$$
(Here $E$ is the ground set of the arithmetic matroid $\cA$, while $\rk$ and $m$ are the rank and the multiplicity functions, respectively).
As the name suggests, $Z_\cA(q,\vv)$ generalizes the polynomials above. It is naturally defined starting from an arithmetic matroid, or more generally from what we call a \emph{pseudo-arithmetic matroid} (see Section \ref{axiomsec}). In fact, our polynomial encodes \emph{all} the structure of the (pseudo-)arithmetic matroid, i.e., it is possible to reconstruct $\cA$ from $Z_\cA(q,\vv)$.

For this polynomial we prove a deletion-contraction recurrence (Lemma \ref{del-con}) and a generalization of Crapo's formula (Theorem \ref{MAC}).

We also give a new proof of the nonnegativity of the coefficients of the (bivariate) arithmetic Tutte polynomial in the more general framework of pseudo-arithmetic matroids (Theorem \ref{nonneg}). 
%Even more, we show that pseudo-arithmetic matroids are the most general objects closed under deletion and contraction and such that the associated arithmetic Tutte polynomials have nonnegative coefficients. 
Moreover, if $\cA$ is represented by a list $\LL$ of elements in a finitely generated abelian group $G$, then in Theorem \ref{geom-int} we give a geometrical interpretation of such coefficients, generalizing various formulae proved in \cite{DM, MoT}.

When $\cA$ is representable, we also provide a Fortuin-Kasteleyn formula for $Z_\cA(q,\vv)$ (Theorem \ref{F-K}),
with applications to arithmetic colorings. This can be seen as a generalization of the ``finite field method'' for computing the characteristic polynomial or the Tutte polynomial of a rational hyperplane arrangement \cite{Ard, Ath,WW}, as well as a generalization of a similar result for toric arrangements \cite{ERS}. We also introduce a generalized ``flow polynomial'' with applications to arithmetic flows (see Theorem \ref{flow}).
%, which allows us to extend the results proved in \cite{DM-G} for graphs with labeled edges (Corollary \ref{graph}). 

Furthermore, we introduce a quasi-polynomial $Z_\LL^P(q,\vv)$ that interpolates between the ordinary and the arithmetic multivariate Tutte polynomial (Theorems \ref{genFK}, \ref{MFK}, \ref{F-K}), and a \emph{Tutte quasi-polynomial} $Q_\LL(x,y)$ that intepolates between the corresponding bivariate polynomials, and specializes to a c\emph{hromatic quasi-polynomial} and to a \emph{flow quasi-polynomial}. The quasi-polynomial $Z_\LL^P(q,\vv)$ is the partition function of a generalized Potts model similar to the one studied by Caracciolo, Sportiello and Sokal, see \cite[Section 3.2]{Sok}.

\bigskip
\section{Arithmetic matroids and multivariate Tutte polynomials}\label{axiomsec}
The notion of an arithmetic matroid tries to capture the linear algebraic \emph{and} arithmetic information contained in a finite list of vectors in $\ZZ^n$. 

Let $\mathbb{N}:=\{0,1,2,\ldots\}$ and $2^E:=\{ A : A \subseteq E\}$, where $E$ is a finite set. We recall that a \emph{matroid},  $\MM$,  on $E$ may be defined via its \emph{rank function}, which is a function  $\rk:2^E\rightarrow \mathbb{N}$ satisfying:
\begin{itemize}
    \item[(R0)] $\rk(\emptyset)=0$.
    \item[(R1)] If $A,B\subseteq E$ and $A\subseteq B$, then $\rk(A)\leq \rk(B)$.
    \item[(R2)] If $A,B\subseteq E$, then $\rk(A\cup B)+\rk(A\cap B)\leq
    \rk(A)+\rk(B)$.
\end{itemize}

We shall now equip a matroid with a \emph{multiplicity function} $m : 2^E \rightarrow \RR$ satisfying certain positivity and/or divisibility properties. 

If $R \subseteq S\subseteq E$, let
$[R,S]:=\{A : R \subseteq A \subseteq S\}$. We say that $[R,S]$ is a \emph{molecule} if $S$ is the disjoint union $S=R\cup F \cup T$ and for each
$A \in [R,S]$
$$
\rk(A) = \rk(R)+ |A \cap F|.
$$
Note that if $[R,S]$ is a molecule and $[R',S'] \subseteq [R,S]$, then $[R',S']$ is a molecule. 
%(see Remark \ref{rem-mole}).
%\begin{definition}
A \emph{pseudo-arithmetic matroid} $\cA=(\MM, m)$ is a matroid $\MM$ equipped with a function
$m:2^E\rightarrow \mathbb{R}$  satisfying
the following axiom:
\begin{itemize}
\item[(P)] If $[R,S]$ is a molecule, then
\begin{equation*}
\rho(R,S):=(-1)^{|T|} \sum_{A \in [R,S]} (-1)^{|S|-|A|}m(A) \geq 0.
\end{equation*}
\end{itemize}
%\end{definition}

A \emph{quasi-arithmetic matroid} $\cA=(\MM, m)$ is a matroid $\MM$ equipped with a function
$m:2^E\rightarrow \mathbb{N}$  satisfying
the following axioms:
\begin{itemize}
       \item[(A1)] For all $A\subseteq E$ and $e\in E$: \begin{align*}
       \text{If }\rk(A \cup \{e\})= \rk(A) \text{, then }&
    m(A\cup\{e\})\text{ divides }m(A)\\
    \text{ otherwise }& m(A)  \text{ divides } m(A\cup\{e\}).
       \end{align*}
    \item[(A2)] If $[R,S]$ is a molecule, then
    $$m(R)m(S) = m(R\cup F)m(R\cup T).$$
   \end{itemize}
An \emph{arithmetic matroid} is a quasi-arithmetic matroid which is also pseudo-arithmetic, i.e., it satisfies (A1), (A2) and (P).

\begin{remark}
The relevance of pseudo-arithmetic matroids is clear from Theorem \ref{nonneg}.  Quasi-arithmetic matroids also arise in combinatorial topology (see \cite[Remark 3.3]{BBC}). As it will be shown later in this paper, axioms (A1) and (A2) are ``algebraic" (see Lemmas \ref{grr}, \ref{molequot}), while axiom (P) has a more geometrical flavor (see Section \ref{geom-repr}).
\end{remark}

The \emph{dual} of an arithmetic (or pseudo-arithmetic) matroid is defined
as the matroid with the same ground set $E$, rank  function defined by
$$\rk^*(A):=|A|-\rk(E)+\rk(E\setminus A)$$ and multiplicity function defined by $m^*(A):=m(E\setminus A)$.
Notice that each of axioms (A1), (A2) and (P) is self-dual, so the dual of an arithmetic matroid is indeed an arithmetic matroid.

Our definition of arithmetic matroid is equivalent to the one given in \cite{DM}. In  \cite{DM}, the function $m$ satisfies five axioms. The first three correspond precisely to (A1) and (A2), while the others are:
  \begin{itemize}
  \item[(4)] If $A\subseteq B\subseteq E$ and $\rk(A)=\rk(B)$, then $\rho(A, B)\geq 0$.
  \item[(5)] If $A\subseteq B\subseteq E$ and $\rk^*(A)=\rk^*(B)$, then $\rho^*(A, B)\geq 0$.
  \end{itemize}
Clearly, axiom $(P)$ implies (4) and (5).
On the other hand, axioms (A2), (4), (5) together imply (P).
Indeed, by axiom (A2)
\begin{align*}
\rho(R,S) &= (-1)^{|T|} \sum_{A_1 \subseteq T, A_2 \subseteq F} (-1)^{|T|+|F|-|A_1|-|A_2|}\frac{m(R\cup A_1)m(R\cup A_2)}{m(R)} \\
               &= \frac{\rho(R, R\cup T)\rho(R, R\cup F)}{m(R)} \geq 0,
\end{align*}
where the last inequality follows from axioms (4) and (5).

The main example of an arithmetic matroid is the one associated to a
finite list of elements of a finitely generated abelian group
$G$, see Section \ref{rep}.

Recall \cite{Cr,Tu} that to each matroid $\MM$ is associated the \emph{Tutte polynomial} %of $\MM$
$$
T_\MM(x,y)= \sum_{A \subseteq E}(x-1)^{\rk(E)-\rk(A)}(y-1)^{|A|-\rk(A)}
$$
that Sokal \cite{Sok} generalized by defining a \emph{multivariate Tutte polynomial} in the variables $q^{-1}$, $\vv= \{v_e\}_{e \in E}$:
$$
Z_\MM(q,\vv) := \sum_{A \subseteq E} q^{-\rk(A)} \prod_{e \in A}v_e.
$$

Similarly, to each arithmetic matroid $\cA$, is associated the
\emph{arithmetic Tutte polynomial} %of $\cA$
$$
M_\cA(x,y)= \sum_{A \subseteq E} m(A)(x-1)^{\rk(E)-\rk(A)}(y-1)^{|A|-\rk(A)},
$$
see \cite{DM,MoT}, that we generalize to a \emph{multivariate arithmetic Tutte polynomial}:
$$
Z_\cA(q,\vv) := \sum_{A \subseteq E} m(A)q^{-\rk(A)} \prod_{e \in A}v_e.
$$
Of course these polynomials are also defined for a pseudo-arithmetic matroid. Note that
\begin{align}%\label{ZtoM}
Z_\MM\left( (x-1)(y-1), y-1\right) &= (x-1)^{-\rk(E)}T_\MM(x,y); \label{ZtoT}\\
Z_\A\left( (x-1)(y-1), y-1\right) &= (x-1)^{-\rk(E)}M_\cA(x,y). \label{ZtoM}
\end{align}
(By $\vv=y-1$ we mean that each variable $v_e$ is evaluated at $y-1$).

\bigskip
\section{Deletion and contraction}\label{last}

Let $\cA$ be an arithmetic matroid. Given an element $e\in
E$, the \textit{deletion} of $\cA$ by $e$ is the arithmetic
matroid, $\cA_1$, on the set
$E_1:=E\setminus \{e\}$, with rank function $\rk_1$ and multiplicity function $m_1$ that are just the
restriction of the corresponding functions of $\cA$.

We also define the \textit{contraction} of $\cA$ by $e$ as
the matroid $\cA_2$ on the set
$E_2:=E\setminus \{e\}=E_1$, with rank function $\rk_2$ given
by $\rk_2(A):=\rk(A\cup \{e\})-\rk(\{e\})$ and multiplicity function given by
$m_2(A):=m(A\cup \{e\})$ for
all $A\subseteq E_2$.

Clearly, the same constructions hold for pseudo-arithmetic matroids.

If an arithmetic matroid $\cA$ is represented by a
list $\mathcal{L}_E$ of elements of a finitely generated abelian group $G$ (see Section \ref{rep}), it is easy to check that the deletion $\cA_1$
is represented by the list $\mathcal{L}_{E_1}$ in $G$, while the contraction $\cA_2$
is represented by the list $E_2:=\{g_a+\langle g_e\rangle : a\in
E\setminus \{e\}\}$ of cosets in $G/\langle g_e\rangle$.

%We say that $e\in E$ is:
%\begin{itemize}
%  \item \textit{free} if both
%$\rk_1(E\setminus\{e\})=\rk(E\setminus\{e\})=\rk(E)-1$ and
%$\rk_2(E\setminus \{e\})=\rk(E)-1$;
%  \item \textit{torsion} if both
%$\rk_1(E\setminus\{e\})=\rk(E)$ and $\rk_2(E\setminus \{e\})=\rk(E)$;
%  \item  \textit{proper} otherwise, i.e. if both
%$\rk_1(E\setminus\{e\})=\rk(E)$ and $\rk_2(E\setminus
%\{e\})=\rk(E)-1$.
%\end{itemize}

%\begin{remark}\label{rem-mole}
%An interval $[R, S]$ is a molecule if and only if, after contracting (the arithmetic matroid defined on) $S$ by all the elements in $R$, each of the remaining elements is either free or torsion.
%\end{remark}

We say that $e\in E$ is:
\begin{itemize}
  \item \emph{free} (or a \emph{coloop}) if both
$$\rk_1(E\setminus\{e\})=\rk(E\setminus\{e\})=\rk(E)-1 \text{ and }\rk_2(E\setminus \{e\})=\rk(E)-1;$$
  \item \emph{torsion} (or \emph{a loop}) if both
$$\rk_1(E\setminus\{e\})=\rk(E)\text{ and }\rk_2(E\setminus \{e\})=\rk(E);$$
  \item  \textit{proper} otherwise, i.e. if both
$$\rk_1(E\setminus\{e\})=\rk(E)\text{ and } \rk_2(E\setminus
\{e\})=\rk(E)-1.$$
\end{itemize}

\begin{remark}\label{rem-mole}
An interval $[R, S]$ is a molecule if and only if 
the matroid defined by contracting all the elements in $R$ and deleting all the elements not in $S$ is such that each remaining element  is either a coloop or a loop. In this sense we use the word ``molecule" in a slightly more general meaning than in \cite{DM}.
\end{remark}

Let $\cA$ be an arithmetic matroid, and let $\cA_1$ and $\cA_2$ be its deletion and contraction by an element $e\in E$.

\begin{lemma}\label{del-con}
$$Z_{\cA}(q,\mathbf{v})=
\left\{
  \begin{array}{ll}
     Z_{\cA_1}(q,\mathbf{v}) + v_e \; Z_{\cA_2}(q,\mathbf{v}), & \hbox{if $e$ is a loop;} \\
    Z_{\cA_1}(q,\mathbf{v}) + (v_e/q)  Z_{\cA_2}(q,\mathbf{v}), & \hbox{otherwise} \\
 \end{array}
\right.
$$
\end{lemma}

\begin{proof}
Decompose $Z_\cA(q,\vv)$ as
$$
Z_\cA(q,\vv) = \sum_{A \subseteq E, e\notin A} m(A)q^{-\rk(A)} \prod_{e \in A}v_e+ \sum_{A \subseteq E, e \in A} m(A)q^{-\rk(A)} \prod_{e \in A}v_e.
$$
If $e\notin A$, then $\rk(A)=\rk_1(A)$. Hence
the first summand equals $Z_{\cA_1}(q,\mathbf{v}).$
On the other hand if $e\in A$, then $\rk(A)=\rk_2(A)+\epsilon$, where $\epsilon=0$ if $e$ is a loop and $\epsilon=1$ otherwise. Thus
\begin{align*} \sum_{A \subseteq E, e \in A} \! \! \! m(A)q^{-\rk(A)} \prod_{e \in A}v_e &=v_e \! \! \sum_{A \subseteq E, e \in A}\! \!  m_2(A\setminus e)q^{-\rk_2(A\setminus e)-\epsilon} \prod_{b \in A\setminus e}v_b \\
&= v_e q^{-\epsilon} Z_{\cA_2}(q,\mathbf{v}).
\end{align*}
\end{proof}

\bigskip
\section{Positivity of coefficients and generalizations of Crapo's formula}\label{craposec}
%It is well known that the coefficients of the Tutte polynomial are positive, and they
%have a nice combinatorial interpretation. In fact, the Tutte
%polynomial embodies two statistics on the list of the bases called
%\textit{internal} and \textit{external activity}.
The following combinatorial interpretation of the coefficients of the Tutte polynomial was proved in \cite{Cr}.

\begin{theorem}[Crapo]\label{Cra}
$$T(x,y)=\sum_{B\in \BB}x^{i(B)}y^{e(B)},$$
where $i(B)$ and $e(B)$ are the internal and external activity defined below.  
\end{theorem}

Let $\BB$ be the set of bases of a matroid on $E$, and let $\BB^*=\{E \setminus B : B \in \BB\}$ the bases of the dual matroid.
Fix a total order on $E$, and let  $B\in \BB$.
An element $e\in E\setminus B$ is \emph{externally active} on $B$
if $e$ is dependent on the list of elements of $B$ following it
(in the total order fixed on $E$). An element $e\in B$ is
\emph{internally active} on $B$ if in the dual matroid $e$ is externally active on the
complement $B^c:=E\setminus B\in \BB^*$. 
In other words, $e\in B$ is internally active if there is no element $f\in E$ such that $e$ follows $f$ and $B\setminus \{e\}\cup\{f\}$ is a basis.
Denote by $E(B)$ the set of externally active elements and by $e(B)$ its cardinality, which is called the
\emph{external activity} of $B$. In the same way,  denote by $I(B)$ the set of internally active elements, and by $i(B)$ its cardinality, which is called the \emph{internal activity} of $B$.

We will now extend Theorem \ref{Cra} in various directions. First we deal with molecules. 

\begin{lemma}\label{lemma-rho2}
Let $[R,S]=[R,R\cup F \cup T]$ be a molecule in a matroid $\MM$, and let $m : E \rightarrow \RR$ be arbitrary. Then
\begin{align*}
& \sum_{A \in [R,S]} m(A)q^{-\rk(A)} \prod_{e \in A} v_e  \\
&= q^{-\rk(S)}\left(\prod_{e \in R\cup F}v_e\right) \sum_{K \subseteq F, L \subseteq T} \rho(R\cup L, S \setminus K)
\prod_{e \in K}\left(\frac q {v_e}+1\right) \prod_{e \in L}(v_e+1). 
\end{align*}
Moreover if $(\MM,m)$ is a quasi-arithmetic matroid, then 
\begin{align*}
&\sum_{A \in [R,S]}m(A)q^{-\rk(A)} \prod_{e \in A} v_e  \\
&= q^{-\rk(S)}\left(\prod_{e \in R\cup F}v_e\right)\left(\sum_{K\subseteq F}\frac {\rho(R, R\cup (F\setminus K))}{m(R)}\prod_{e \in K}\left(\frac q {v_e}+1\right) \right) \\
&\times \left(\sum_{L\subseteq T}\rho(R\cup L, R\cup T) \prod_{e \in L}(v_e+1)\right). 
\end{align*}
Note that ${\rho(R, R\cup (F\setminus K))}/{m(R)}$ is an integer by Axiom (A1). 
\end{lemma}

\begin{proof}
Note that 
\begin{align*}
&\sum_{A_1 \subseteq T, A_2 \subseteq F}m(R\cup A_1 \cup A_2)\prod_{e \in F \setminus A_2} (v_e-1) \prod_{e \in A_1}(v_e-1) \\
&=\sum_{A_1 \subseteq T, A_2 \subseteq F} m(R\cup A_1 \cup A_2)\sum_{K \subseteq F \setminus A_2, L \subseteq A_1}\vv^{K}\vv^{L}(-1)^{|A_1|+|A_2|+|F|+|K|+|L|}\\
&= \sum_{K,L} \vv^{K}\vv^{L} \sum_{L \subseteq A_1 \subseteq T, A_2 \subseteq F \setminus K}(-1)^{|A_1|+|A_2|+|F|+|K|+|L|}m(R\cup A_1 \cup A_2)\\
&= \sum_{K\subseteq F,L \subseteq T} \rho(R \cup L, S \setminus K) \vv^{K}\vv^{L}, 
\end{align*}
where $\vv^A= \prod_{e \in A} v_e$. The first formula now follows by a change of variables since $\rk(R \cup K \cup L)= \rk(R)+|K|$. 

The formula for quasi-arithmetic matroids follows from the first using Axiom (A2): If $A_1 \subseteq T\setminus L$ and $A_2 \subseteq F \setminus K$, then 
$$
m(R\cup L \cup A_1 \cup A_2) = \frac {m(R\cup L \cup A_1)m(R \cup A_2)}{m(R)}.
$$
But then 
$$
\rho(R \cup L, S \setminus K)= \frac {\rho(R\cup L, R\cup T)\rho(R, R\cup F \setminus K)}{m(R)},
$$
so the sum factors as claimed. 
\end{proof}

Let us formulate a bivariate version of Lemma \ref{lemma-rho2} for reference.

\begin{lemma}\label{lemma-rho}
Let $[R,S]=[R,R\cup F \cup T]$ be a molecule in a matroid $\MM$, and let $m : E \rightarrow \RR$ be arbitrary. Then
\begin{align*}
& \sum_{A \in [R,S]} m(A) (x-1)^{\rk(S)-\rk(A)} (y-1)^{|A|-\rk(A)}\\
&= \sum_{K\subseteq F,L \subseteq T} \rho(R \cup L, S \setminus K) x^{|K|}y^{|L|}.
\end{align*}
Moreover if $(\MM,m)$ is a  quasi-arithmetic matroid, then 
\begin{align*}
&\sum_{A \in [R,S]} m(A) (x-1)^{\rk(S)-\rk(A)} (y-1)^{|A|-\rk(A)}\\
&=\left(\sum_{K\subseteq F}\frac {\rho(R, R\cup (F\setminus K))}{m(R)}x^{|K|}\right)\left(\sum_{L\subseteq T}\rho(R\cup L, R\cup T) y^{|L|}\right). 
\end{align*}
\end{lemma}
%\begin{proof}
%For the first equality,
%\begin{align*}
%& \sum_{A \in [R,S]} m(A) (x-1)^{\rk(S)-\rk(A)} (y-1)^{|A|-\rk(A)} \\
%&= \sum_{A_1 \subseteq T, A_2 \subseteq F}m(R\cup A_1 \cup A_2)(x-1)^{|F \setminus A_2|}(y-1)^{|A_1|} \\
%&=\sum_{A_1 \subseteq T, A_2 \subseteq F} m(R\cup A_1 \cup A_2)\sum_{K \subseteq F \setminus A_2, L \subseteq A_1}x^{|K|}y^{|L|}(-1)^{|A_1|+|A_2|+|F|+|K|+|L|}\\
%&= \sum_{K,L} x^{|K|}y^{|L|} \sum_{L \subseteq A_1 \subseteq T, A_2 \subseteq F \setminus K}(-1)^{|A_1|+|A_2|+|F|+|K|+|L|}m(R\cup A_1 \cup A_2)\\
%&= \sum_{K\subseteq F,L \subseteq T} \rho(R \cup L, S \setminus K) x^{|K|}y^{|L|}.
%\end{align*}
%The second equality follows from the first using Axiom $(A_2)$: If $A_1 \subseteq T\setminus L$ and $A_2 \subseteq F \setminus K$, then 
%$$
%m(R\cup L \cup A_1 \cup A_2) = \frac {m(R\cup L \cup A_1)m(R \cup A_2)}{m(R)}.
%$$
%But then 
%$$
%\rho(R \cup L, S \setminus K)= \frac {\rho(R\cup L, R\cup T)\rho(R, R\cup F \setminus K)}{m(R)},
%$$
%so the sum factors as claimed. 
%\end{proof}

The next proposition describes how $2^E$ may be partitioned into molecules. For a proof see \cite{Cr} and  \cite{Bjorn}.

\begin{proposition}\label{bj}
Suppose that $\MM$ is a matroid with ground set $E$ and set of bases $\BB$. Then:
\begin{itemize}
  \item[\emph{i)}] $2^E$ is the disjoint union
$$
2^E= \bigsqcup_{B \in \BB} [B \setminus \IA(B), B \cup \EA(B)];
$$
  \item[\emph{ii)}] for each $B \in \BB$, $[B \setminus \IA(B), B \cup \EA(B)]$ is a molecule with $F=\IA(B)$ and $T=\EA(B)$.
\end{itemize}
\end{proposition}

\begin{theorem}\label{nonneg}
The coefficients of the arithmetic Tutte polynomial $M_\cA(x,y)$ of a pseudo-arithmetic matroid $\cA$ are nonnegative.
Moreover, pseudo-arithmetic matroids are the most general objects closed under deletion and contraction and such that the associated arithmetic Tutte polynomials have nonnegative coefficients.
\end{theorem}
\begin{proof}
Consider the molecule $[R,S]= [B \setminus \IA(B), B \cup \EA(B)]$.  By the first equality proved in Lemma \ref{lemma-rho} the polynomial
$$
\sum_{A \in [R,S]} m(A) (x-1)^{\rk(S)-\rk(A)} (y-1)^{|A|-\rk(A)}
$$
has nonnegative coefficients. By Proposition \ref{bj}, this implies that $M_\cA(x,y)$ has non-negative coefficients.

For the second statement, notice that if $[R,S]$ is a molecule we may contract the elements in $R$ and delete the elements in
$E \setminus S$. Again by the first equality of Lemma \ref{lemma-rho} we see that the constant coefficient  of the arithmetic Tutte polynomial of the corresponding matroid is $\rho(R,S)$.
\end{proof}

Crapo's formula has been extended in \cite{DM} to arithmetic matroids.
We shall now see how Lemma \ref{lemma-rho2} gives a multivariate version of this, and extends it to  pseudo-arithmetic matroids with integral multiplicities.

If $\cA$ is a pseudo-arithmetic matroid with integral multiplicities, let $\widetilde{\mathcal{B}}$ be the multi-set with elements of the form $(B,C)$ where $B$ is a basis, 
$C \subseteq E(B) \cup I(B)$, and where the pair 
$(B,C)$ appears exactly 
$$\rho\big( (B\cup C)\setminus I(B)\, , \, (B\setminus C) \cup E(B)\big)$$ times.   

\begin{theorem}\label{MAC}If $\cA$ is a pseudo-arithmetic matroid with integral multiplicities, then
$$
Z_\cA(q,\vv)= q^{-\rk(E)}\!\!\!\sum_{(B,C)\in
\widetilde{\mathcal{B}}}\prod_{b\in B}v_b \!\!\!\prod_{e\in E(B,C)}\!\!\!(v_e+1)\!\!\!\prod_{i\in I(B,C)}\!\!\!\left(\frac{q}{v_i}+1\right), 
$$
where $E(B, C)= E(B)\cap C$ and $I(B, C)= I(B)\cap C$. 

In particular, 
$$M_\cA(x,y)=\sum_{(B,C)\in
\widetilde{\mathcal{B}}}x^{i(B,C)}y^{e(B,C)}, $$
where $e(B,C)=|E(B,C)|$ and $i(B,C)=|I(B,C)|$.
\end{theorem}

\begin{proof}
The proof follows straightforwardly from  Proposition \ref{bj}, Lemma \ref{lemma-rho2} and  Lemma~\ref{lemma-rho}. 
\end{proof}

For ordinary matroids we have the following consequence of Theorem \ref{MAC}. 
\begin{corollary}\label{Cra-Sok}
If $\MM$ is a matroid, then 
$$
Z_\MM(q,\vv)= q^{-\rk(E)}\!\sum_{B\in
\mathcal{B}}\prod_{b\in B}v_b \!\!\!\prod_{e\in E(B)}\!\!\!(v_e+1)\!\!\!\prod_{i\in I(B)}\!\!\!\left(\frac{q}{v_i}+1\right). 
$$
\end{corollary}

We define the \textit{external activity polynomial} of a basis
$B\subseteq E$ to be
$$
E_B(y):=\sum_{T\supseteq B}\rho(T)y^{e(B,T)}, 
$$
where $\rho(T) := \rho(T,E)$ and $e(B,T):= |T \cap E(B)|$. 
%Dually, define the \textit{internal activity polynomial} of
%the basis $B$ as $E
%$$
%E_{B^c}^*(x):=\sum_{\widetilde{T}\supseteq
%B^c}\rho^*(\widetilde{T})x^{e^*(B^c,\widetilde{T})}.
%$$
%DEFINE $\rho$ and $\rho^*$!

As a consequence of Lemma \ref{lemma-rho}, we get a simpler proof of a theorem proved in \cite[Section 4]{DM}.

\begin{theorem} \label{sez4}
The arithmetic Tutte polynomial of an arithmetic matroid $\cA$ decomposes as
$$
M_\cA(x,y)=\mathop{\sum_{B\in \BB}}\frac{1}{m(B)}E^*_{B^c}(x) E_{B}(y),
$$
where $E^*_{B^c}(x)$ is the external activity polynomial of the basis $B^c$ in the dual arithmetic matroid. 
\end{theorem}
\begin{proof}
Consider the molecule  $[R,R \cup F \cup T]$ where $R= B \setminus I(B)$, $F=I(B)$ and $T=E(B)$. 
By an elementary but somewhat tedious computation which is left to the reader one sees that 
$$
\frac {m(R)}{m(B)}E_B(y)= \sum_{L\subseteq T}\rho(R\cup L, R\cup T) y^{|L|}
$$
and 
$$
E^*_{B^c}(x)=\sum_{K\subseteq F} {\rho(R, R\cup (F\setminus K))}x^{|K|}.
$$
Hence the proof follows from Proposition \ref{bj} and Lemma \ref{lemma-rho}. 
\end{proof}

Furthermore, if $\cA$ is representable, the polynomials $E_B(y)$ and $E^*_{B^c}(x)$ have nice combinatorial interpretations, which we will give in Section \ref{geom-repr}.

\bigskip
\section{Representable arithmetic matroids}\label{rep}
Let $\LL=(g_e)_{e \in E}$ be a finite list of vectors in a finitely generated abelian group
$G$. By a \emph{list} we will always mean an element $(g_e)_{e \in E}$ of $G^E$, and we do not require $E$ to be totally ordered.  Recall that a finitely generated abelian group $G$ is isomorphic to $G_f\oplus G_t$,
where $G_t$ is finite and
$G_f$ is free abelian, i.e., it is isomorphic to $\mathbb{Z}^r$ for
some $r\geq 0$. Then $G_t$ is called the \emph{torsion} of $G$ and $r:=\rk(G)$ is the \emph{rank} of $G$. A matroid with ground set $E$ is naturally defined by its rank function $\rk : 2^E \rightarrow \NN$ defined by $\rk(A)= \rk(\langle \LL_A \rangle )$, where $\LL_A= (g_e)_{e \in A}$ and $\langle \LL_A \rangle$ is the subgroup generated by $\LL_A$. In addition to the matroid structure, $\LL$ carries arithmetic information which is encoded as multiplicities.
For $A\subseteq E$, let $H_A$ be the maximal subgroup of $G$
such that $\langle \LL_A\rangle\leq H_A$ and $|H_A: \langle
\LL_A\rangle|<\infty$, where $|H_A: \langle \LL_A\rangle|$ denotes the
index (as subgroup) of $\langle \LL_A\rangle$ in $H_A$. The
multiplicity $m(A)$ is defined as $m(A):=|H_A: \langle \LL_A\rangle|$.  For our purposes it is natural to express the multiplicity in other terms:  Let $G_A := \big( G/\langle \LL_A \rangle\big)_t$ be the torsion subgroup of $G/\langle \LL_A \rangle$, and define $m(A) := |G_A|$. It is straightforward to see that the two definitions agree. 

The fact that the multiplicity function just described satisfies the original five axioms for an arithmetic matroid (and hence our axioms) has been verified in \cite{DM}. It also follows from the results in this paper, namely Lemma \ref{grr} and \ref{molequot} and Section \ref{geom-repr}.
\begin{remark}\emph{The most familiar situation is when $G=\ZZ^n$. However if we want to allow for deletion and contraction (see Section \ref{last}) we need the above more general setup.}
\end{remark}

We will denote by $\MM_\LL$ the matroid determined by $\LL$, and by $\A_\LL = (\MM_\LL, m)$ the \emph{arithmetic matroid} determined by $\LL$. We say that an arithmetic matroid is \emph{representable} if it comes from such a list.

\begin{lemma}\label{grr}
Let $A \subset E$ and $e \in E \setminus A$. 
If $\rk(A) < \rk(A\cup\{e\})$, then there exists a group monomorphism $G_A \hookrightarrow G_{A\cup \{e\}}$.
If $\rk(A) = \rk(A\cup\{e\})$, then there exists a group epimorphism $G_A \twoheadrightarrow G_{A\cup \{e\}}$.
\end{lemma}
\begin{proof}
Since
$$
G/\langle \LL_{A} \cup \{g_e\}\rangle \cong (G/ \langle \LL_{A} \rangle)/ \langle g_e +  \langle \LL_{A} \rangle\rangle,
$$
we may assume that $A = \emptyset$. Let $\pi : G \rightarrow G/ \langle g_e \rangle$ be the natural projection. Then $\pi$ is a homomorphism between the torsion parts.

Suppose that $g_e$ is not a torsion element.  If $t \in G$ is a torsion element for which $\pi(t)= 0$, then $t \in \langle g_e \rangle$ so that $t=mg_e$ for some $m \in \ZZ$. Since $g_e$ is not a torsion element $m=0$, and hence $t=0$. Hence $\pi$ is an injection when restricted to the torsion part.

Suppose that $g_e$ is a torsion element. The torsion part of $G / \langle g_e \rangle$ consists of all elements of the form $g+\langle g_e \rangle$ where $mg \in \langle g_e \rangle$ for some $m \in \ZZ$. But $mg \in \langle g_e \rangle$ for some $m \in \ZZ$ if and only if $g$ is a torsion element. Hence $\pi$ is an epimorphism between the torsion parts.
\end{proof}

Furthermore, if $A\subseteq B\subseteq E$ are such that $\rk(A)=\rk(B)$, then the composition morphism $G_A \twoheadrightarrow G_B$ does not depend on the order chosen on $B\setminus A$. In the same way, if $\rk^*(A)=\rk^*(B)$, the composition morphism $G_A \hookrightarrow G_B$ does not depend on the order chosen on $B\setminus A$.
%Then for every molecule $[R,S]$ we have morphisms $G_R \hookrightarrow G_{R\cup F}\twoheadrightarrow G_S$ and $G_R \twoheadrightarrow G_{R\cup T}\hookrightarrow G_S$.
Then for every molecule $[R,S]$ we have a commutative diagram
\begin{equation*}
  \xymatrix@R+2em@C+2em{
  G_{R} \ar[r] \ar[d] & G_{R\cup F} \ar[d] \\
  G_{R\cup T} \ar[r] & G_{R\cup F\cup T}
  }
 \end{equation*}
in which the horizontal arrows are injective and the vertical arrows are surjective.

\begin{lemma}\label{molequot}
For each molecule $[R, R\cup F \cup T]$
$$
\frac {G_{R\cup F}}{G_R} \cong \frac {G_{R\cup F \cup T}}{G_{R\cup T}}.
$$
\end{lemma}
\begin{proof}
As in Lemma \ref{grr} we may assume that $R=\emptyset$. By Lemma \ref{grr} we have an epimorphism $G_{F} \twoheadrightarrow G_{F \cup T}$
and hence an epimorphism $$\phi : G_{F} \twoheadrightarrow G_{F \cup T}/{G_{T}}.$$ Now $g + \langle \LL_F\rangle \in \ker(\phi)$ if and only if $mg \in \langle \LL_T \rangle$
if and only if $g$ is a torsion element, that is, $g \in G_\emptyset$. Hence $G_F/G_{\emptyset} \cong G_{F \cup T}/{G_{T}}$.
\end{proof}

\bigskip
\section{A geometric interpretation of the arithmetic Tutte polynomial}\label{geom-repr}
In this section the arithmetic matroids are representable. Let $[R,R\cup F\cup T]$ be a molecule. Here we give geometric interpretations of the polynomials
$$
f_{[R,R\cup F]}(x):=\sum_{K\subseteq F}\frac{\rho(R, R\cup (F\setminus K))}{m(R)}x^{|K|}
$$
and 
$$
g_{[R, R\cup T]}(x):= \sum_{L\subseteq T}\rho(R\cup L, R\cup T) x^{|L|}.
$$
By Lemma \ref{lemma-rho} and Proposition \ref{bj} this gives a geometric interpretation of the arithmetic Tutte polynomial (see Theorem \ref{geom-int}). 

It is not hard to see that there is a list of vectors $\LL' = (f_e)_{e \in F} \subset \ZZ^n$ such that 
$$
\frac G {\langle \LL_{R\cup K} \rangle} \cong \frac {\ZZ^n} {\langle \LL'_{K} \rangle} \oplus G_R,
$$
for all $K \subseteq F$.  Hence  $$\rho(R, R \cup K)= m(R) \rho'(\emptyset, K),$$ where the accent refers to the list $\LL'$ in $\ZZ^n$.  Let $\ZT(K)$ be the zonotope generated by $(f_e)_{e \in K}$, i.e., 
$$
 \ZT(K) := \left\{ \sum_{e \in E} t_e f_e : 0 \leq t_e \leq 1 \mbox{ for all } e \in E\right\}. 
 $$

By e.g. Proposition \ref{Ehrart},
$\rho'(\emptyset, K)$ counts the number of interior points in $\ZT(K)$. Let now 
$\PP[R,R\cup F]$ be the set of integer points in  the semi-open zonotope $$\left\{ \sum_{e \in F} t_e f_e : 0\leq t_e <1 \mbox{ for all } e \in E\right\} ,$$ and let 
$\iota(p)$, where $p \in \PP[R,R\cup F]$ be defined as the number of $e \in F$ for which $t_e=0$ in $p= \sum_{e \in F} t_e f_e$.  
\begin{lemma}
Let $[R,R\cup F]$ be a molecule and let $\PP= \PP[R,R\cup F]$ be as above. Then 
$$
f_{[R,R\cup F]}(x)= \sum_{p \in \PP} x^{\iota(p)}.
$$
\end{lemma}
\begin{proof}
Each point $p \in \PP[R,R\cup F]$ is an interior point in $\ZT(F \setminus K)$ for a unique set $K$. The number of zero coordinates is then $|K|$. 
\end{proof}

Let $\mathcal{T}(R\cup T)$ be the generalized toric arrangement defined by $R\cup T$; this is a set of subgroups (hence submanifolds) in the abelian compact Lie group $\Hom(G, \mathbb{S}^1)$. Namely, we have one subgroup $H_e$ for each $e\in R\cup T$, having codimension $0$ if $e$ is torsion and $1$ otherwise. Let $\CT(R)$ be the set of connected components of the subgroup $\cap_{e\in R}H_e$.
Define a function
$\eta: \CT(R)\rightarrow \NN$ as follows: For each $c\in \CT(R)$, let $\eta(c)$ be the number of elements $e\in T$ such that $H_e\supseteq c$.

\begin{lemma}
With definitions as above, 
$$
g_{[R, R\cup T]}(x)= \sum_{c \in \CT(R)} x^{\eta(c)}.
$$
\end{lemma}
\begin{proof}
It is known that $m(A) = |\CT(A)|$, see \cite{MoT}. If $R \subseteq A \subseteq D \subseteq R \cup T$, then $\CT(A) \supseteq \CT(D)$. For $A \in [R, R\cup T]$, let 
\begin{equation}\label{ct}
\widetilde{\CT}(A) :=\CT(A) \setminus \!\!\!\!\!\bigcup_{A \subsetneq D \subseteq R \cup T}\CT(D).
\end{equation}
Then $$\CT(R) = \bigcup_{L \subseteq T}\widetilde{\CT}(R\cup L)$$ where the union is disjoint. By inclusion--exclusion  
$$\rho(R\cup L, R \cup T) = |\widetilde{\CT}(R\cup L)|.$$ If $c \in \widetilde{\CT}(R\cup L)$, then the indices $e \in T$ for which $H_e \supseteq c$ are exactly those in $L$. This proves the lemma. 
\end{proof}
Using Lemma \ref{lemma-rho} and Proposition \ref{bj} we may now deduce the following geometric interpretation of the arithmetic Tutte polynomial. Let $B$ be a basis. 
To simplify the notation we set
\begin{align*}
\PP_B &:=\PP[B\setminus I(B), B],\\
\CT_B &:= \CT(B\setminus I(B)),\\
\eta(c) &:= \eta_{E(B)}(c).
\end{align*}
To summarize, recall that $\PP_B$ is the set of integer points of the semiopen zonotope defined by $F=I(B)$, while $\CT_B$ is a set of layers of the generalized toric arrangement defined by $R\cup T= (B\setminus I(B))\cup E(B)$.
\begin{theorem}\label{geom-int}
Let $\LL$ be a list of vectors in a finitely generated abelian group $G$. With definitions as above,
$$
M_{\LL}(x,y) = \sum_{B \in \BB} \left(\sum_{p \in \PP_B} x^{\iota(p)} \right)\left(\sum_{c \in \CT_B} y^{\eta(c)} \right).
$$
\end{theorem}

\bigskip
\section{A Fortuin-Kasteleyn quasi-polynomial}\label{kastsec}

Let $\GG=(V,E)$ be a finite graph with vertex set $V$ and edge set $E$.
In \cite{Sok} the multivariate Tutte polynomial of a graph was defined as
$$
Z_\GG(q,\vv) := \sum_{A \subseteq E} q^{k(A)} \prod_{e \in A}v_e,
$$
where $k(A)$ denotes the number of connected components in the subgraph $(V,A)$. The multivariate Tutte polynomial has an interpretation in statistical physics as the \emph{partition function of the $q$-state Potts model}.
\begin{theorem}[Fortuin--Kasteleyn]\label{FK-rep}
For any positive integer $q$,
$$
Z_\GG(q,\vv)= \sum_{\sigma : V \rightarrow [q]} \prod_{e=ij \in E} (1+v_e\delta(\sigma(i), \sigma(j))),
$$
where $\delta$ is the Kronecker delta and $[q]:=\{1,\ldots, q\}$.
\end{theorem}
Theorem \ref{FK-rep} is known as the \emph{Fortuin--Kasteleyn representation} of the $q$-state Potts model.
Our main goal of this section is to generalize this theorem to lists of vectors in finitely generated abelian groups.

We define here a generalization of the Potts model which is similar to the one studied by Caracciolo, Sportiello and Sokal, see \cite[Section 3.2]{Sok}.
Let $\LL=(g_e)_{e \in E}$ be a list of elements in a finitely generated abelian group $G$, and let $H$ be a finite abelian group. Then
$$
Z_\LL(G, H,\vv) := \sum_{\phi \in \Hom(G, H)} \prod_{e \in E}(1+v_{e}\delta(\phi(g_e),0)).
$$

The special case when $H=\ZZ_q:= \ZZ / q\ZZ$ will be particularly interesting, and we set
$$
 Z_\LL^P(q,\vv) := \sum_{\phi \in \Hom(G, \ZZ_q)} \prod_{e \in E}(1+v_{e}\delta(\phi(g_e),0)).
$$

We will prove that $Z_\LL^P(q,\vv)$ is a quasi-polynomial in $q$ that interpolates between the arithmetic multivariate Tutte polynomial $Z_{\cA_\LL}(q,\vv)$ of the arithmetic matroid $\cA_\LL$ and the multivariate Tutte polynomial $Z_{\MM_\LL}(q,\vv)$ of its underlying matroid $\MM_\LL$.

\begin{remark}\label{not-ari}
Notice that $Z_\LL^P(q,\vv)$ is \emph{not} an invariant of the arithmetic matroid. For instance, the empty list in $\ZZ_4$ defines the same arithmetic matroid as the empty list in $\ZZ_2 \oplus \ZZ_2$, but $Z_\LL^P(2,\vv)$ is equal to $2$ in the former case, and to $4$ in the latter.
\end{remark}

First a simple lemma. For a group $G$ and integer $q$, let $qG:= \{ qh : h \in G\}$ where $\ZZ$ acts on $G$ in the usual way.
\begin{lemma}\label{hom-q}
If $G$  is a finitely generated abelian group of rank $r$, then
$$
\Hom(G,\ZZ_q) \cong  G/qG  \cong \ZZ_q^r \oplus G_t/qG_t.
$$
%Hence if $qG_t = (0)$, then
%$$
%\Hom(G,\ZZ_q) \cong  \ZZ_q^r \oplus G_t.
%$$
\end{lemma}
\begin{proof}
The lemma follows easily from the structure theorem of finitely generated abelian groups  by observing that
$\Hom(G, \ZZ_q) \cong \Hom(G/qG, \ZZ_q)$, $\Hom(\ZZ_a, \ZZ_{ab}) \cong \ZZ_a$,  $\Hom(\ZZ, \ZZ_{q}) \cong \ZZ_q$, and  $$\Hom(G_1 \oplus G_2, \ZZ_{q}) \cong  \Hom(G_1, \ZZ_{q}) \oplus \Hom(G_2, \ZZ_{q}).$$
\end{proof}
%
%
%We define three multivariate polynomials associated to a list $\LL=(g_e)_{e \in E}$ of elements in a finitely generated abelian group $G$.
%\begin{align*}
%Z_\LL^M(q,\vv) &:= \sum_{A \subseteq E}q^{-\rk(A)} \prod_{e \in A}v_e,  \\
%Z_\LL^A(q,\vv) &:= \sum_{A \subseteq E}m(A)q^{-\rk(A)} \prod_{e \in A}v_e,  \\
% Z_\LL^P(q,\vv) &:= \sum_{\phi \in \Hom(G, \ZZ_q)} \prod_{e \in E}(1+v_{e}\delta(\phi(g_e),0)),
%\end{align*}
%where $\ZZ_q:= \ZZ / q\ZZ$. Here $M, A$ and $P$ stand for matroid, arithmetic and Potts, respectively.

%Define the \emph{arithmetic multivariate Tutte polynomial} of a list $\LL=(g_e)_{e \in E}$ of elements in a finitely generated abelian group $G$ to be
%$$
%Z_\LL(q,\vv) := \sum_{A \subseteq E} m(A)q^{\rk(G)-\rk(A)} \prod_{e \in A}v_e.
%$$
%%
%%
%When $m(A)=1$ for all $A \subseteq E$ this is just the multivariate Tutte polynomial. Furthermore, this polynomial generalizes the (univariate) arithmetic Tutte polynomial. In fact we have:

%\section{Fortuin--Kasteleyn representation.}
\begin{theorem}\label{genFK}
If $H \cong \oplus_{i=1}^k \ZZ_{q_i}$ and $q=|H|$, then
$$
Z_\LL(G, H,\vv) = q^{\rk(G)}\sum_{A \subseteq E}q^{-\rk(A)} \prod_{i=1}^k \frac{m(A)}{|q_iG_A|}\prod_{e \in A}v_e.
$$
\end{theorem}

\begin{proof}
Assume first that $H=\ZZ_q$. Expand the product in the sum to obtain
\begin{align*}
 \sum_{\phi \in \Hom(G, \ZZ_q)} \prod_{e \in E}(1+v_{e}\delta(\phi(g_e),0)) &= \sum_{A \subseteq E} |\{\phi : \langle \LL_A \rangle \subseteq \ker(\phi)\}| \prod_{e \in A}v_e \\
&= \sum_{A \subseteq E} |\Hom( G/\langle  \LL_A \rangle, \ZZ_q)| \prod_{e \in A}v_e .
\end{align*}
By Lemma~\ref{hom-q},
\begin{equation}\label{alt-m}
|\Hom(G/ \langle \LL_A \rangle, \ZZ_q)|=|G_A/qG_A| q^{\rk(G)-\rk(A)} =\frac {m(A)}{|qG_A|}q^{\rk(G)-\rk(A)}.
\end{equation}
The proof for general $H$ follows by observing that if $H=H_1\oplus H_2$, then
$$
\Hom(G/ \langle \LL_A \rangle, H_1\oplus H_2) \cong \Hom(G/ \langle \LL_A \rangle, H_1) \oplus \Hom(G/ \langle \LL_A \rangle, H_2).
$$
\end{proof}

\begin{remark}
\emph{Since $(q+|G|)G=qG$ holds for any finite group $G$ it follows that $Z_\LL^P(q,\vv)$ is a quasi-polynomial in $q$.}
\end{remark}

 Let $\LCM(\LL)$ denote the least common multiple of all $m(B)$, where $B \subseteq E$ is a basis. In the case when $G=\ZZ^n$ and $\LL$ has rank $r$, then $\LCM(\LL)$ equals the least common multiple of all non-zero $r\times r$ minors of $\LL$. Define two subsets of $\ZZ_+:=\{n \in \ZZ: n>0\}$ as follows
\begin{align*}
\ZZ_M(\LL) &:= \{q \in \ZZ_+ : \text{GCD}(q, \LCM(\LL))=1\},  \\
\ZZ_A(\LL) &:= \{q \in \ZZ_+ : qG_B=(0) \mbox{ for all bases } B \subseteq E \}.
\end{align*}
For example if $q$ is a multiple of $\LCM(\LL)$, then $q \in \ZZ_A(\LL)$.

\begin{theorem}\label{MFK}
Let $|H|=q$. Then $q \in \ZZ_M(\LL)$ if and only if
$$
Z_\LL(G, H,\vv) = q^{\rk(G)} Z_{\MM_\LL}(q,\vv),
$$
as a polynomial in $\vv$.
\end{theorem}
\begin{proof}%[Proof of Theorem \ref{MFK}]
By Theorem \ref{genFK}, the equality holds if and only if $q_iG_A=G_A$ for all $A \subseteq E$ and $1 \leq i \leq k$, which is equivalent to that $\text{GCD}(q_i, |G_A|)=1$ for all $A \subseteq E$. By Lemma \ref{grr}, for each $A \subseteq E$, there is a basis $B \subseteq E$ for which $|G_A|$ divides $|G_B|$. Hence we only have to check the condition for bases.
\end{proof}

Note that when $\LL \subset \ZZ^n =G$ is a totally unimodular matrix, i.e., $m(B)=1$ for all bases $B$, then $\ZZ_M(\LL)=\ZZ_A(\LL)=\ZZ_+$. Thus Theorem \ref{MFK} extends \cite[Theorem 3.1]{Sok} and can also be seen as a refinement of the ``finite field method'' to compute the characteristic polynomial of a hyperplane arrangement, see \cite{Ath} (or its Tutte polynomial, see \cite{Ard,WW}). The proof of the following theorem follows immediately from Theorem \ref{genFK}:

\begin{theorem}\label{F-K}
Let $q$ be a positive integer. Then $q \in \ZZ_A(\LL)$ if and only if
\begin{equation}\label{just}
Z_\LL^P(q,\vv)= q^{\rk(G)} Z_{\A_\LL}(q,\vv),
\end{equation}
as a polynomial in $\vv$. In particular if $q$ is a multiple of $\LCM(\LL)$, then \eqref{just} holds. 
\end{theorem}

Theorem \ref{F-K} is a refinement of the finite field method to compute the characteristic polynomial of a toric arrangement, see \cite{ERS}.

\begin{example}
\emph{Let us see why Theorem \ref{FK-rep} is a special case of Theorem \ref{F-K}. Let $\GG=(V,E)$ be a graph on $V=[n]$. Let further $G=\ZZ^n$
and $\LL=(g_e)_{e \in E}$ where $g_e$ is the vector with $i$th coordinate $1$ and $j$th coordinate $-1$ and the other coordinates $0$, where $e=\{i,j\}$ and $i<j$.  Then $m(A) =1$ for all $A$ since the matrix $\LL$ is totally unimodular, and hence $G_A=(0)$ for all $A \subseteq E$. Moreover $\Hom(G,\ZZ_q)= \ZZ_q^V$, so  Theorem \ref{FK-rep} follows.}
\end{example}

\begin{example}\emph{
Let $G= \ZZ^3$ and
$$
\LL= \left[ \begin{array}{ccccccc}
1&1&0&0&0&1&1 \\
0&1&1&1&0&0&1 \\
0&0&0&1&1&1&1
\end{array}\right].
$$
Then $\MM_\LL$ is the non--Fano matroid $F_7^-$, which is not a regular matroid (see for example \cite{Ox}). The multiplicities of the bases are given by the absolute values of the nonzero $3 \times 3$ minors and are thus equal to $1$ or $2$. Hence for $q$ a positive integer
$$
Z_\LL^P(q,\vv)= q^3\begin{cases} Z_{F_7^-}(q,\vv), \mbox{ if $q$ is odd, }\\
Z_{\A_\LL}(q,\vv), \mbox{ if $q$ is even. }
\end{cases}
$$
}
\end{example}

%\bigskip
%\section{The Tutte quasi-polynomial}\label{Tutte-quasi}
In view of Theorem \ref{genFK} it is natural to define a \emph{Tutte quasi-polynomial} by 

$$
{Q}_\LL(x,y) :=
\sum_{A \subseteq E} \frac{|G_A|}{|(x-1)(y-1) G_A|}(x-1)^{\rk(E)-\rk(A)}(y-1)^{|A|-\rk(A)}.
$$
By Theorem \ref{genFK}, we have an analogue of \eqref{ZtoT} and \eqref{ZtoM}:
\begin{equation}\label{ZtoQ}
Z^P_\LL\left( (x-1)(y-1), y-1\right) = (x-1)^{-\rk(E)}Q_\LL(x,y).
\end{equation}
Hence ${Q}_\LL(x,y)$ is a quasi-polynomial in $q=(x-1)(y-1)$, which by Theorems \ref{MFK}, \ref{F-K} coincides with $M_{\LL}(x,y)$  when $q\in \ZZ_A(\LL)$ and with $T_{\LL}(x,y)$ when $q\in \ZZ_M(\LL)$.

From Remark \ref{not-ari} we see that ${Q}_\LL(x,y)$ is an invariant of the list $\LL$, and not of the arithmetic matroid.

\bigskip
\section{Generalized flows}\label{flowsec}

For $A \subseteq E$ and $q \in \ZZ_+$ define a homomorphism  $\Psi_A^q : \ZZ_q^A \rightarrow G/qG$ by
$$
\Psi_A^q(\phi) = \sum_{e \in A} \phi(e)g_e + qG.
$$
By analogy with $q$-flows in graphs (see Section \ref{Se-gra}), an element $\phi \in \ker(\Psi_E^q)$ will be called a $(\LL,q)$-\emph{flow}. If in addition $\phi(e) \neq 0$ for all
$e \in E$, $\phi$ is called \emph{nowhere-zero}. Hence a nowhere-zero $(\LL,q)$-flow is a map $\phi : E \rightarrow \ZZ_q \setminus \{0\}$ for which
$
 \sum_{e \in E} \phi(e)g_e = 0
$
in $G/qG$.

The \emph{multivariate arithmetic flow polynomial} of $\LL$ is defined by
$$
F_\LL(q,\vv):= \sum_{\phi \in \ker(\Psi_E^q)} \prod_{e \in E} (1 +v_e \delta(\phi(e),0)).
$$

\begin{theorem}\label{flow}
For positive integers $q$,
$$
F_\LL(q,\vv)= q^{-\rk(G)}\frac {|qG_t|}{m(\emptyset)}\left( \prod_{e \in E}v_e \right)Z_\LL^P(q, q/\vv),
$$
where $q/\vv := \{q/v_e\}_{e \in E}$.
Moreover,
\begin{enumerate}
\item If $q \in \ZZ_M(\LL)$, then
$$
F_\LL(q,\vv)= \left( \prod_{e \in E}v_e \right)Z_{\MM_\LL}(q, q/\vv).
$$
\item If $q \in \ZZ_A(\LL)$, then
$$
F_\LL(q,\vv)= \frac {1}{m(\emptyset)} \left( \prod_{e \in E}v_e \right)Z_{\A_\LL}(q, q/\vv).
$$
\end{enumerate}
\end{theorem}
\begin{proof}Expand the product in the sum to get
$$
\sum_{\phi \in \ker(\Psi_E^q)} \prod_{e \in E} (1 +v_e \delta(\phi(e),0))= \sum_{A \subseteq E} |\ker(\Psi_A^q)|\prod_{b \in E \setminus A}v_b.
$$
By Lemma \ref{hom-q} and the third isomorphism theorem
$$
\Hom(G/ \langle \LL_A \rangle, \ZZ_q) \cong \frac {G/\langle \LL_A \rangle}{q \left(G/\langle \LL_A \rangle\right)}
\cong \frac {G/\langle \LL_A \rangle}{(qG+\langle \LL_A \rangle )/\langle \LL_A \rangle } \cong \frac {G}{qG+\langle \LL_A \rangle}.
$$
Also
$$
\Im(\Psi_A^q) \cong {(qG+\langle \LL_A \rangle)/qG},
$$
and hence by the  third isomorphism theorem again
$$
\Hom(G/ \langle \LL_A \rangle, \ZZ_q)\cong \frac {G}{qG+\langle \LL_A \rangle} \cong \frac {G/qG}{(qG+\langle \LL_A\rangle)/qG}.
$$
By \eqref{alt-m} and Lemma \ref{hom-q}
$$
\frac{m(A)}{|G_A|}q^{\rk(G)-\rk(A)} = \frac {|G/qG|}{|(qG+\langle \LL_A\rangle)/qG|}= $$
$$=\frac{ m(\emptyset)q^{\rk(G)}}{|qG_\emptyset||\Im(\Psi_A^q)|}= \frac{m(\emptyset)}{|qG_\emptyset|}q^{\rk(G)-|A|}|\ker(\Psi_A^q)|.
$$
Hence
$$
F_\LL(q,\vv)= \sum_{A \subseteq E} |\ker(\Psi_A^q)|\prod_{b \in E \setminus A}v_b =
\frac {|qG_t|} {m(\emptyset)}\sum_{A \subseteq E} \frac{m(A)}{|qG_A|}q^{|A|-\rk(A)} \prod_{b \in E \setminus A}v_b.
$$
\end{proof}

\bigskip
\section{Chromatic quasi-polynomials and flow quasi-polynomials}\label{Se-gra}

Let us see how colorings and flows on graphs may be generalized using the ideas in the previous sections. Some of the following observations were also made in \cite{DM-G}, where the results were obtained by deletion-contraction arguments. 

Let $\Di=(V,E)$ be a directed graph without loops and assign positive integer weights $\{ w_e \}_{e \in E}$ to the edges. Define a list of vectors $\LL=(g_e)_{e \in E}$  in $G:= \ZZ^V$ as follows. If $e$ is an edge from $i$ to $j$, let $g_e$ be the vector with entry $j$ equal to $w_e$, entry $i$ equal to $-w_e$, and 
the other entries equal to $0$. Then we define a \emph{proper $(\LL,q)$-coloring} 
%(as defined in Section \ref{kastsec}) 
as a map $\phi : V \rightarrow \ZZ_q$ such that 
$$
w_e \phi(i) \neq w_e \phi(j),
$$
whenever $e$ is an edge between $i$ and $j$. We define the \emph{chromatic quasi-polynomial} $\chi_\LL(q)$ as the function that counts such colorings. The fact that this is actually a quasi-polynomial, and that it is independent on the orientation of the graph, are consequences of the next theorem.

Dually, the flow quasi-polynomial $\chi^*_\LL(q)$ counts maps $\phi : E \rightarrow \ZZ_q\setminus \{0\}$ such that the weighted Kirchhoff laws hold: for each $i \in V$
$$
\sum_{i \stackrel e {\rightarrow} } \phi(e)w_e = \sum_{\stackrel e {\rightarrow} i} \phi(e)w_e,
$$ 
where $i \stackrel e {\rightarrow}$ means that $e$ is an edge from $i$ to some other vertex. 

We will show that the quasi-polynomials above are specializations of the Tutte quasi-polynomial. This fact holds more generally for any list $\LL$ of elements of a finitely generated abelian groups $G$. By analogy with the graphic case, we call an element $\phi \in \Hom(G,\ZZ_q)$ a \emph{proper $(\LL,q)$-coloring} if $\phi(g_e) \neq 0$ for all $e \in E$, and we denote by $\chi_\LL(q)$ the number of such colorings. Notice that by definition this is equal to the evaluation of $Z^P_\LL(q, \vv)$ at $v_e=-1$ for all $e\in E$, that we denote by $Z^P_\LL(q, -1)$. We call $\chi_\LL(q)$ the \emph{chromatic quasi-polynomial}. 

Dually, denote by $\chi_\LL^*(q)$ the number of nowhere-zero $(\LL,q)$-flows. By Theorem \ref{flow} this is a quasi-polynomial in $q$ which we call the \emph{flow quasi-polynomial}.

We have the following result, generalizing a famous theorem of Tutte \cite{Tu}, and also its analogue in the arithmetic setting \cite{DM-G}:

\begin{theorem}\label{color-cor}\label{flow-cor}
Let $q$ be a positive integer. Then
\begin{align}
\chi_\LL(q) &=(-1)^{\rk(E)}q^{\rk(G)-\rk(E)}Q_\LL(1-q,0), \quad \mbox{ and } \label{l1} \\
\chi_\LL^*(q)&=(-1)^{|E|-\rk(E)}\left(m(\emptyset)\right)^{-1}Q_\LL(0,1-q). \label{l2}
\end{align}
In particular, if $q \in \ZZ_M(\LL)$, then
\begin{align}
\chi_\LL(q) &= (-1)^{\rk(E)}q^{\rk(G)-\rk(E)}T_\LL(1-q,0), \quad \mbox{ and }  \label{l3} \\
\chi_\LL^*(q)&= (-1)^{|E|-\rk(E)}T_\LL(0,1-q). \label{l4}
\end{align}
If, on the other hand, $q \in \ZZ_A(\LL)$, then
\begin{align}
\chi_\LL(q) &= (-1)^{\rk(E)}q^{\rk(G)-\rk(E)}M_\LL(1-q,0), \quad \mbox{ and }  \label{l5} \\
\chi_\LL^*(q)&= (-1)^{|E|-\rk(E)}M_\LL(0,1-q). \label{l6}
\end{align}
%\item $$
%\chi_\LL(q)=(-1)^{\rk(E)} q^{\rk(G)-\rk(E)}M_\LL(1-q,0).
%$$
%and
%\item $$
%\chi_\LL^*(q)=(-1)^{|E|-\rk(E)}\left(m(\emptyset)\right)^{-1}M_\LL(0,1-q)
%$$
%\end{enumerate}
\end{theorem}
\begin{proof} 

The first statement follows immediately from \eqref{ZtoQ}, since $\chi_\LL(q)=Z^P_\LL(q, -1)$. Likewise \eqref{l2} follows from Theorem \ref{flow} and \eqref{ZtoQ} because $\chi_\LL^*(q)$ is equal to $F_\LL(q, -1)$. 
The other statements are consequences of the first two and Theorems \ref{MFK}, \ref{F-K} and \ref{flow}.

\end{proof}

One may of course generalize the setting for graphs by for example allowing the nonzero entries of $g_e$ to not be equal in magnitude. We choose not to develop this direction here. 

\bigskip
\section{A multivariate Ehrhart polynomial}\label{se-Ehrart}
In this section we assume that $G$ has no torsion, i.e., $G\cong \ZZ^r$ for some $r$. Then, a list $\LL=(g_e)_{e\in E}$ in $G$ defines a \emph{zonotope} 
$$\mathcal{Z}(\LL):= \left\{\sum_{e\in E}s_e g_e:  0\leq s_e \leq 1 \mbox{ for all } e\in E\right\}.$$

Consider the polynomial 
$$
E_\LL(\vv) := Z_{\A_\LL}(q,q\vv)\Big|_{q=0}= \sum_{A} m(A) \prod_{e \in A}v_e,$$
where the sum is over all independent sets $A$ in $\MM$.  For  $\mathbf{k}=(k_e)_{e \in E} \in \mathbb{N}^E$ let $\mathbf{k}\cdot \LL= (k_eg_e)_{e \in E}$. 

\begin{proposition}\label{Ehrart}Let $\mathbf{k}=(k_e)_{e \in E} \in \mathbb{N}^E$. 

\begin{enumerate}
  \item $E_\LL(\kk)$ is equal to the number of integer points of the zonotope $\mathcal{Z}(\mathbf{k}\cdot \LL)$;
  \item $E_\LL(-\kk)$ is equal to the number of integer points in the interior of $\mathcal{Z}(\mathbf{k}\cdot \LL)$;
  \item the univariate polynomial $t \mapsto E_\LL(t, \ldots, t)$ is the Ehrhart polynomial of $\mathcal{Z}(\LL)$;
  \item the highest degree part of $E_\LL(\vv)$ is the mixed volume of the 1-dimensional polytopes $\mathcal{Z}(\{g_1\}), \ldots, \mathcal{Z}(\{g_n\})$.
\end{enumerate}
\end{proposition}

\begin{proof}
It is known (see for instance \cite{DM-E}) that the number of integer points in $\mathcal{Z}(\LL)$ is
$$\sum \{ m(A): A \subseteq E \mbox{ is independent}\}. $$
However if $A$ is independent, then the multiplicity of $A$ in $\kk \cdot \LL$ is equal to $m(A) \prod_{e \in A}k_e$. 
This proves the first statement and the third. The second is then a consequence of the well-known reciprocity theorem for the Ehrhart polynomial. The fourth statement is then also clear.
\end{proof}

The theorem above generalizes a result proved in \cite{DM-E}. In fact, the polynomial
$E_\LL(\vv)$ is a multivariate version of the Ehrhart polynomial. It may be considered as a \emph{mixed Ehrhart polynomial}.

For example if $\LL=\{(3,0), (0,2), (1,1)\}$, then
$$E_\LL(v_1, v_2, v_3)= 1+3v_1+2v_2+v_3+6v_1v_2+3v_1v_3+2v_1v_2.$$


\begin{thebibliography}{99}
\bibitem{Ard} \textsc{F. Ardila}, \textit{Computing the Tutte polynomial of a hyperplane arrangement},
Pacific J. Math.: 230, 1--26, 2007.
\bibitem{Ath} \textsc{C. A. Athanasiadis}, \textit{Characteristic polynomials of subspace arrangements and finite fields}, Adv. Math.: 122, 193--233, 1996.
\bibitem{BBC} \textsc{C. Bajo, B. Burdick, S. Chmutov}, \textit{On the Tutte-Krushkal-Renardy polynomial for cell complexes}, arXiv:1204.3563v1 [math.CO].
\bibitem{Bjorn} \textsc{A. Bj\"orner}, \textit{The homology and shellability of matroids and geometric lattices. Matroid applications}, Encyclopedia Math. Appl., 40, Cambridge Univ. Press, Cambridge, 226--283, 1992.
\bibitem{Cr} \textsc{H. Crapo}, \textit{The Tutte polynomial}, Aequationes Math., 3: 211--229, 1969.
\bibitem{DM-E} \textsc{M. D'Adderio, L. Moci}, \textit{Ehrhart polynomial of the zonotope and arithmetic Tutte polynomial}, European J. of Combinatorics 33 (2012), 1479--1483.
\bibitem{DM} \textsc{M. D'Adderio, L. Moci}, \textit{Arithmetic matroids, Tutte polynomial, and toric arrangements}, 	arXiv:1105.3220 [math.CO], to appear in Advances in Mathematics.
\bibitem{DM-G} \textsc{M. D'Adderio, L. Moci}, \textit{Graph colorings, flows and arithmetic Tutte polynomial}, J. of Combinatorial Theory, Series A, 120 (2013) 11--27
\bibitem{ERS} \textsc{R. Ehrenborg, M. Readdy, M. Slone}, \textit{Affine and toric hyperplane arrangements}, Discrete Comput. Geom.: 41, 481--512, 2009.
\bibitem{MoT} \textsc{L. Moci}, \textit{A Tutte polynomial for toric arrangements}, Trans. Amer. Math. Soc. 364 (2012), 1067-1088.
%\bibitem{MoF} \textsc{L. Moci}, \textit{Zonotopes, toric arrangements, and generalized Tutte polynomials}, Proceedings FPSAC 2010 (DMTCS).
\bibitem{Ox} \textsc{J. G. Oxley}, \textit{Matroid Theory}, Oxford University Press, Oxford 1992.
\bibitem{Sok} \textsc{A. Sokal}, \textit{The multivariate Tutte polynomial (alias Potts model) for graphs and matroids}, Surveys in combinatorics 2005, 173-226, London Math. Soc. Lecture Note Ser., 327, Cambridge Univ. Press, Cambridge, 2005.
\bibitem{Tu} \textsc{W. T. Tutte}, \textit{A contribution to the theory of chromatic polynomials},  Canadian J. Math., 6: 80-91, 1954
\bibitem{WW}\textsc{D. J. A. Welsh, G. P. Whittle}, \textit{Arrangements, channel assignments, and associated polynomials}, Adv. in Appl. Math. 23: 375--406, 1999.
\end{thebibliography}
\end{document}